\documentclass[11pt,a4paper]{amsart}

\usepackage{amsmath, amssymb, amsthm}
\usepackage{mathrsfs} 
\usepackage{hyperref} 
\usepackage{enumitem} 
\usepackage{geometry} 
\usepackage[utf8]{inputenc}

\newtheorem{theorem}{Theorem}[section]
\newtheorem{proposition}[theorem]{Proposition}

\newtheorem{corollary}[theorem]{Corollary}

\theoremstyle{definition}
\newtheorem{definition}[theorem]{Definition}
\newtheorem{example}[theorem]{Example}

\theoremstyle{remark}
\newtheorem{remark}[theorem]{Remark}

\newcommand{\Ham}{\operatorname{Ham}}
\newcommand{\Diff}{\operatorname{Diff}}

\newcommand{\ImMap}{\operatorname{Im}}
\newcommand{\id}{\operatorname{id}}
\newcommand{\Vol}{\operatorname{Vol}}
\newcommand{\osc}{\operatorname{osc}}
\newcommand{\Cal}{\operatorname{Cal}}
\newcommand{\Lomega}{\mathcal{L}^\omega}

\newcommand{\domega}{d^\omega}
\newcommand{\domegastar}{d^{\omega*}}

\title[Geometric Properties and Flux of LCS Diffeomorphisms]{GEOMETRIC PROPERTIES AND FLUX OF LOCALLY CONFORMALLY SYMPLECTIC DIFFEOMORPHISMS}

\author{S. Tchuiaga}
\address{Department of Mathematics, University of Buea, P.O. Box 63, Buea, South West Region, Cameroon}
\email{tchuiaga.kameni@ubuea.cm}

\author{F. Balibuno}
\address{Department of Mathematics and Computer Science, Faculty of Sciences, University of Kinshasa, P.O. Box 190, Kinshasa XI, DRC}
\email{fidele.balibuno@unikin.ac.cd}

\date{April 30, 2025}

\subjclass[2020]{Primary 53D12, 53D05; Secondary 53D22, 58D05, 53D35.}

\keywords{Locally conformally symplectic manifold, Lichnerowicz cohomology, Flux homomorphism, Hamiltonian LCS diffeomorphism, Mapping torus, Lagrangian neighborhood, Calabi invariant, Energy-Capacity Inequality, Splitting Theorem, Deformation Rigidity, Non-displaceability, Fragmentation.}

\begin{document}

	\begin{abstract}
		
		We investigate the geometric and topological properties of the group of locally conformally symplectic (LCS) diffeomorphisms, utilizing the LCS flux homomorphism defined by S. Haller. By analyzing the flux map from the universal cover of the identity component $(\ker \Phi)_0$ to the first Lichnerowicz cohomology group $H_\omega^1(M)$, we establish a short exact sequence characterizing the Hamiltonian subgroup $\Ham_\Omega(M)$ and provide conditions for its topological splitting as a semidirect product. We develop LCS analogues of fundamental symplectic results, including a Weinstein neighborhood theorem, a flux rigidity theorem for homotopies, and a characterization of LCS structures on mapping tori. A central theme of this work is the influence of the Hodge decomposition of the Lee form $\omega = dh + l$. In the exact case ($l=0$), we utilize the global conformal equivalence to symplectic structures to establish energy-capacity inequalities, an LCS Hofer metric, and non-displaceability results. We explicitly analyze the relationship between the LCS Calabi invariant and its symplectic counterpart, showing they are controlled by a multiplicative factor depending on the conformal weight. For the general non-exact case ($l \neq 0$), we introduce a Twisted Calabi invariant that captures the interaction between Hamiltonian dynamics and the harmonic component of the Lee form.
		
	\end{abstract}
	
	\maketitle
	
	\tableofcontents
	
	\section{Introduction}
	
	Locally conformally symplectic (LCS) geometry provides a natural generalization of symplectic geometry. An LCS manifold $(M, \Omega, \omega)$ consists of a $2n$-dimensional manifold $M$, a non-degenerate 2-form $\Omega$, and a closed 1-form $\omega$ (the Lee form) satisfying the structure equation $d\Omega = -\omega \wedge \Omega$. Introduced by H.C. Lee \cite{Lee1943}, this geometry connects various fields including contact and complex geometry. A fundamental dichotomy in LCS geometry arises from the Hodge decomposition of the closed Lee form $\omega$. On a compact oriented Riemannian manifold, any closed 1-form $\omega$ can be uniquely decomposed as $\omega = dh + l$, where $h$ is a smooth function and $l$ is a harmonic 1-form. When $l = 0$, $\omega$ is exact ($\omega = dh$), and the LCS structure is globally conformally equivalent to a symplectic structure via $\Omega_h = e^h\Omega$. In this exact case, many LCS concepts reduce to their symplectic counterparts. When $l \neq 0$, the harmonic component represents a nontrivial de Rham cohomology class, and the structure is truly locally conformally symplectic without being globally conformally symplectic. This decomposition provides a natural measure of how "far" an LCS structure is from being symplectic and fundamentally influences the geometry and topology of LCS diffeomorphism groups. A central theme in geometric structures is understanding the group of diffeomorphisms that preserve them. For LCS manifolds, these are the LCS diffeomorphisms $g$, characterized by the conditions $g^*\Omega = a^{-1}\Omega$ and $g^*\omega = \omega + d(\ln a)$ for some smooth positive function $a$. The study of the structure, topology, and geometry of the group of compactly supported LCS diffeomorphisms, $\Diff_\Omega(M)$, and its identity component $\Diff_{\Omega,0}(M)$, is fundamental.	A powerful tool for studying the topology of diffeomorphism groups is the flux homomorphism. Building upon the work of S. Haller \cite{Haller2005}, who defined the LCS flux homomorphism $\widetilde{\Psi}$ mapping paths in a specific subgroup $(\ker \Phi)_0 \subseteq \Diff_{\Omega,0}(M)$ to the first Lichnerowicz cohomology group $H_\omega^1(M)$ (Theorem \ref{thm:haller_flux}), this paper explores numerous consequences of this construction. The Hodge decomposition $\omega = dh + l$ significantly affects the target space $H_\omega^1(M)$: in the exact case ($l=0$), we have $H_\omega^1(M) \cong H_{dR}^1(M)$ via $\alpha \mapsto e^h\alpha$, while in the non-exact case, $H_\omega^1(M)$ differs from de Rham cohomology group.\\
	
	We present the short exact sequence associated with the flux (Proposition \ref{prop:exact_sequence}):
	\[
	1 \longrightarrow \Ham_\Omega(M) \xrightarrow{i} (\ker \Phi)_0 \xrightarrow{\mathcal{F}} H_\omega^1(M)/\Delta \longrightarrow 1,
	\]
	where $\Ham_\Omega(M)$ is the group of Hamiltonian LCS diffeomorphisms (defined as the kernel of the induced flux $\mathcal{F}$), and $\Delta$ is the discrete period group. This sequence reveals how the Hamiltonian subgroup sits inside $(\ker \Phi)_0$. We further investigate conditions under which this sequence splits topologically as a semidirect product (Theorem \ref{thm:splitting}).\\

	Beyond this structural result, we present several geometric theorems within the LCS context, many of which are analogues of well-known results in symplectic geometry and are connected to the concept of flux or Hamiltonian dynamics. The Hodge decomposition $\omega = dh + l$ clarifies which results hold generally for all LCS structures versus those specific to the exact case:
	\begin{itemize}
		\item An LCS version of the Weinstein Lagrangian neighborhood theorem (Theorem \ref{thm:weinstein}).
		\item A flux rigidity phenomenon (Theorem \ref{thm:flux_rigidity}, with a full proof provided).
		\item A characterization stating that the mapping torus of an LCS diffeomorphism $g \in (\ker \Phi)_0$ admits an LCS structure compatible with the fibration if and only if the flux of $g$ vanishes (Theorem \ref{thm:mapping_torus}).
		\item Properties of the LCS Calabi invariant, energy, and Hofer metric \emph{only in the exact case $\omega = dh$} (Theorems \ref{thm:calabi_exact}, \ref{thm:energy_exact}, \ref{thm:hofer_exact}).
		\item Deformation rigidity for zero-flux isotopies (Theorem \ref{thm:deformation_rigidity}).
		\item Fragmentation property for Hamiltonian LCS diffeomorphisms based on \cite{Banyaga1978} (Theorem \ref{thm:fragmentation}).
		\item An LCS non-displaceability theorem for Lagrangian submanifolds \emph{in the exact case $\omega = dh$} (Theorem \ref{thm:non_displaceability}).
	\end{itemize}
	
	Collectively, these results develop tools for analyzing the structure of the LCS diffeomorphism group and related geometric objects like Lagrangian submanifolds, extending concepts from symplectic geometry to the locally conformal setting via the lens of the flux homomorphism and clarifying the role of the harmonic component of the Lee form.
	
	\subsection*{Organization of the Paper}
	The paper is organized as follows. Section 2 reviews the foundational definitions of LCS manifolds, the Morse-Novikov (Lichnerowicz) cohomology, and the crucial Hodge decomposition of the Lee form into exact and harmonic components. Section 3 focuses on the group structure of LCS diffeomorphisms; we detail Haller's flux homomorphism, derive the associated short exact sequence, prove the splitting theorem, and present geometric applications including the LCS Weinstein theorem and flux rigidity. Section 4 is dedicated to the Hamiltonian subgroup $\Ham_\Omega(M)$. Here, we distinguish between the exact case—where we establish the Hofer metric, non-displaceability, and the relationship between the LCS and symplectic Calabi invariants—and the non-exact case, where we introduce the Twisted Calabi invariant and discuss topological properties such as deformation rigidity and fragmentation. Section 5 illustrates these concepts through explicit calculations on the Kodaira-Thurston manifold, highlighting the distinct behaviors of non-exact LCS structures. Finally, Section 6 summarizes the results and outlines future research directions.

	\section{Preliminaries on Locally Conformally Symplectic Geometry}
	
	This section reviews the fundamental definitions and properties of locally conformally symplectic (LCS) manifolds and related concepts.
	
	\begin{definition}[LCS Manifold]
		A \emph{locally conformally symplectic (LCS) manifold} is a triple $(M, \Omega, \omega)$, where $M$ is a $2n$-dimensional manifold, $\omega$ is a closed 1-form (the \emph{Lee form}), and $\Omega$ is a non-degenerate 2-form satisfying
		\[
		d\Omega = -\omega \wedge \Omega.
		\]
	\end{definition}
	
	When $\omega = 0$, $(M, \Omega)$ becomes a symplectic manifold. The non-degeneracy of $\Omega$ induces a bundle isomorphism $TM \to T^*M$ via $X \mapsto i_X \Omega$.
	
	\begin{definition}[Morse-Novikov Differential]
		Let $\omega$ be a closed 1-form on $M$. The \emph{Morse-Novikov differential operator} (or \emph{Lichnerowicz differential}) is the map $d^\omega : \Omega^p(M) \to \Omega^{p+1}(M)$ defined by
		\[
		d^\omega \alpha := d\alpha + \omega \wedge \alpha.
		\]
		This operator satisfies:
		\begin{enumerate}
			\item $d^\omega$ is $\mathbb{R}$-linear.
			\item $d^\omega \circ d^\omega = 0$.
			\item For $\alpha \in \Omega^p(M)$ and $\beta \in \Omega^q(M)$, $d^\omega(\alpha \wedge \beta) = (d^\omega \alpha) \wedge \beta + (-1)^p \alpha \wedge (d^\omega \beta)$.
			\item (Pullback Property) For a smooth map $g : N \to M$, $g^*(d^\omega \alpha) = d^{g^*\omega}(g^*\alpha)$.
		\end{enumerate}
	\end{definition}
	
	\begin{definition}[$d^\omega$-Closed and $d^\omega$-Exact Forms]
		A differential $p$-form $\alpha$ is \emph{$d^\omega$-closed} if $d^\omega \alpha = 0$ and \emph{$d^\omega$-exact} if there exists a $(p-1)$-form $\beta$ such that $\alpha = d^\omega \beta$.
	\end{definition}
	
	\begin{definition}[Lichnerowicz Cohomology]
		The space of $d^\omega$-cocycles is
		\[
		Z_\omega^p(M) := \{ \alpha \in \Omega^p(M) \mid d^\omega \alpha = 0 \},
		\]
		and the space of $d^\omega$-coboundaries is
		\[
		B_\omega^p(M) := \{ d^\omega \beta \mid \beta \in \Omega^{p-1}(M) \}.
		\]
		The $p$-th \emph{Lichnerowicz cohomology group} (or \emph{Morse-Novikov cohomology}) is the quotient space
		\[
		H_\omega^p(M) := Z_\omega^p(M) / B_\omega^p(M).
		\]
		The compactly supported version is denoted by $H_{\omega,c}^p(M)$. Note that $H_\omega^1(M)$ is the target space of the LCS flux homomorphism.
	\end{definition}
	
	\begin{remark}[$d^\omega$-Closedness of $\Omega$]
		Note that $d^\omega \Omega = d\Omega + \omega \wedge \Omega = (-\omega \wedge \Omega) + (\omega \wedge \Omega) = 0$. The LCS form $\Omega$ is always $d^\omega$-closed.
	\end{remark}

	\begin{remark}[Geometric Interpretation]
		The decomposition $\omega = dh + l$ has significant geometric implications:
		\begin{enumerate}
			\item In the exact case ($\omega = dh$), the LCS structure is globally conformally equivalent to a symplectic structure: $\Omega_h := e^h\Omega$ satisfies $d\Omega_h = 0$.
			\item In the non-exact case ($l \neq 0$), no such global conformal factor exists. The harmonic component $l$ represents a nontrivial de Rham cohomology class $[l] \in H^1_{dR}(M)$.
			\item The harmonic part $l$ provides a natural measure of how "far" an LCS structure is from being globally conformally symplectic.
		\end{enumerate}
	\end{remark}
	
	\begin{proposition}[Lichnerowicz vs. de Rham Cohomology]
		Let $(M, \Omega, \omega)$ be an LCS manifold with Hodge decomposition $\omega = dh + l$.
		\begin{enumerate}
			\item In the exact case ($l=0$, $\omega = dh$), the map $\Psi_h : \Omega^p(M) \to \Omega^p(M)$ defined by $\Psi_h(\alpha) = e^h \alpha$ induces an isomorphism of cohomology groups:
			\[
			[\Psi_h]^* : H_\omega^p(M) \xrightarrow{\cong} H_{dR}^p(M), \quad [\alpha]_\omega \mapsto [e^h \alpha]_{dR}.
			\]
			
			\item In the non-exact case ($l \neq 0$), no such global isomorphism exists. However, locally on contractible open sets where $\omega$ is exact, the Lichnerowicz cohomology is locally isomorphic to de Rham cohomology.
		\end{enumerate}
	\end{proposition}
	
	\begin{proof}
		(1) For $\omega = dh$: We compute $d(\Psi_h(\alpha)) = d(e^h \alpha) = e^h dh \wedge \alpha + e^h d\alpha = e^h(d\alpha + dh \wedge \alpha) = e^h(d^\omega \alpha)$. Thus $d^\omega \alpha = 0$ if and only if $d(\Psi_h(\alpha)) = 0$. Also, $\alpha = d^\omega \beta = d\beta + dh \wedge \beta$ maps to $\Psi_h(\alpha) = e^h(d\beta + dh \wedge \beta) = d(e^h \beta)$. Hence $\Psi_h$ induces an isomorphism on cohomology.
		
		(2) For $\omega = dh + l$ with $l$ harmonic and nonzero: The map $\alpha \mapsto e^h \alpha$ sends $$d^\omega \alpha = d\alpha + (dh + l) \wedge \alpha$$ to $e^h(d\alpha + l \wedge \alpha) + e^h dh \wedge \alpha.$ This is not simply $d(e^h \alpha)$ due to the $e^h l \wedge \alpha$ term, so no global isomorphism exists.
	\end{proof}
	
	\subsection{Lichnerowicz Laplacian and Twisted Hodge Theory}
	
	To analyze the flux homomorphism more precisely, we develop a twisted Hodge theory for Lichnerowicz cohomology.
	
	\begin{definition}[Lichnerowicz Codifferential and Laplacian]
		Let $(M,g)$ be a Riemannian manifold with Lee form $\omega$. The \emph{Lichnerowicz codifferential} $\domegastar : \Omega^p(M) \to \Omega^{p-1}(M)$ is the formal $L^2$-adjoint of $\domega$, given by
		\[
		\domegastar \alpha = d^* \alpha + \iota_{\omega^\sharp} \alpha,
		\]
		where $\omega^\sharp$ is the vector field dual to $\omega$ via $g$, and $\iota$ denotes interior product. The \emph{Lichnerowicz Laplacian} is defined as
		\[
		\Delta_\omega = \domega\domegastar + \domegastar\domega.
		\]
		A form $\alpha$ is called \emph{$\Delta_\omega$-harmonic} if $\Delta_\omega \alpha = 0$.
	\end{definition}
	
	\begin{theorem}[Twisted Hodge Decomposition]
		Let $(M,g)$ be a compact oriented Riemannian manifold with closed Lee form $\omega$. There exists a unique $L^2$-orthogonal decomposition:
		\[
		\Omega^p(M) = \domega(\Omega^{p-1}(M)) \oplus \domegastar(\Omega^{p+1}(M)) \oplus \mathcal{H}_\omega^p(M),
		\]
		where $\mathcal{H}_\omega^p(M) = \{ \alpha \in \Omega^p(M) \mid \Delta_\omega \alpha = 0 \}$. Consequently, every cohomology class in $H_\omega^p(M)$ contains a unique $\Delta_\omega$-harmonic representative.
	\end{theorem}
	
	\begin{proof}
		The proof follows standard Hodge theory arguments adapted to the twisted operators. The key observation is that $\Delta_\omega$ is an elliptic operator, as it differs from the standard Hodge Laplacian $\Delta = dd^* + d^*d$ by lower-order terms involving $\omega$. Elliptic regularity and Fredholm theory imply the decomposition.
	\end{proof}
	
	\begin{corollary}[Harmonic Representatives for Flux]
		Every flux class in $H_\omega^1(M)$ has a unique $\Delta_\omega$-harmonic representative. This provides a canonical way to represent flux classes and analyze the period group $\Delta$.
	\end{corollary}
	
	\subsection{LCS Diffeomorphisms and Vector Fields}
	
	\begin{definition}[LCS Diffeomorphism]
		A diffeomorphism $g : M \to M$ is \emph{locally conformally symplectic (LCS)} if there exists a smooth function $a \in C^\infty(M, \mathbb{R}_{>0})$ such that
		\[
		g^*\Omega = a^{-1}\Omega, \quad g^*\omega = \omega + d(\ln a).
		\]
		The function $a$ is called the \emph{conformality factor} of $g$.
	\end{definition}
	
	We denote by $\Diff_\Omega(M)$ the group of compactly supported LCS diffeomorphisms and by $\Diff_{\Omega,0}(M)$ its identity component. The Lie algebra of $\Diff_\Omega(M)$ is
	\[
	\mathcal{X}_{c,\omega}(M) := \{ X \in \mathcal{X}_c(M) \mid \exists f \in C^\infty_c(M) \text{ s.t. } L_X\Omega = -f\Omega, \ L_X\omega = df \}.
	\]
	For $X \in \mathcal{X}_{c,\omega}(M)$, the function $f$ is the infinitesimal conformality factor. Let $\Phi : \mathcal{X}_{c,\omega}(M) \to H_\omega^0(M)$ be defined as in \cite{Haller2005}. Note that $(\ker \Phi)_0$ is the identity component of the subgroup of $\Diff_{\Omega,0}(M)$ whose Lie algebra at the identity is $\ker \Phi$.
	
	\subsection{Twisted Lie Derivative and Characterization of LCS Vector Fields}
	
	\begin{definition}[Twisted Lie Derivative]
		For a vector field $X$ on an LCS manifold $(M, \Omega, \omega)$, the \emph{twisted Lie derivative} $\Lomega_X$ is defined via Cartan's magic formula with $d$ replaced by $\domega$:
		\[
		\Lomega_X \alpha := \domega(i_X \alpha) + i_X(\domega \alpha),
		\]
		for any differential form $\alpha$. Equivalently,
		\[
		\Lomega_X \alpha = \mathcal{L}_X \alpha + \omega(X)\alpha.
		\]
	\end{definition}
	
	\begin{proposition}[Properties of $\Lomega$]
		The twisted Lie derivative satisfies:
		\begin{enumerate}
			\item $\Lomega_X$ is a derivation: $\Lomega_X(\alpha \wedge \beta) = (\Lomega_X \alpha) \wedge \beta + \alpha \wedge (\Lomega_X \beta)$.
			\item Commutation with $\domega$: $[\Lomega_X, \domega] = 0$.
			\item For functions $f$: $\Lomega_X f = X(f)$.
			\item For 1-forms $\alpha$: $\Lomega_X \alpha = \mathcal{L}_X \alpha + \omega(X)\alpha$.
		\end{enumerate}
	\end{proposition}
	
	\begin{theorem}[Characterization of LCS Vector Fields]
		Let $(M, \Omega, \omega)$ be an LCS manifold and $X$ a vector field. The following are equivalent:
		\begin{enumerate}
			\item $X$ is an LCS vector field (i.e., its flow consists of LCS diffeomorphisms).
			\item There exists a constant $c \in \mathbb{R}$ such that $\Lomega_X \Omega = c\Omega$.
			\item $\mathcal{L}_X \Omega = -f\Omega$ and $\mathcal{L}_X \omega = df$ for some function $f$, with $c = 0$ if and only if $f = \omega(X)$.
		\end{enumerate}
		Moreover:
		\begin{itemize}
			\item $X \in \ker \Phi$ (strictly LCS) if and only if $\Lomega_X \Omega = 0$.
			\item $X$ is Hamiltonian if and only if $i_X \Omega = \domega H$ for some function $H$.
		\end{itemize}
	\end{theorem}
	
	\begin{proof}
		The equivalence follows from direct computation using the definitions. For an LCS vector field $X$, its flow $\phi_t$ satisfies $\phi_t^*\Omega = a_t^{-1}\Omega$ for some positive function $a_t$. Differentiating at $t=0$ gives $\mathcal{L}_X \Omega = -\dot{a}_0 \Omega$. Since $\phi_t^*\omega = \omega + d(\ln a_t)$, we get $\mathcal{L}_X \omega = d(\dot{a}_0)$. Setting $f = \dot{a}_0$, we have $\mathcal{L}_X \Omega = -f\Omega$ and $\mathcal{L}_X \omega = df$. Then
		\[
		\Lomega_X \Omega = \mathcal{L}_X \Omega + \omega(X)\Omega = (-f + \omega(X))\Omega.
		\]
		Thus $\Lomega_X \Omega = c\Omega$ with $c = \omega(X) - f$. For $X \in \ker \Phi$, we have $f = \omega(X)$, so $c = 0$.
	\end{proof}
	
	\begin{definition}[LCS Lagrangian Submanifold]
		Let $(M, \Omega, \omega)$ be an LCS manifold of dimension $2n$. An immersion $i: L \to M$ of an $n$-dimensional manifold $L$ is an \emph{LCS Lagrangian immersion} if $i^*\Omega = 0$. A submanifold $L \subset M$ is \emph{LCS Lagrangian} if the inclusion map is an LCS Lagrangian immersion.
	\end{definition}
	
	\begin{example}[Cotangent Bundle LCS Structure] \label{ex:cotangent}
		Let $L$ be a manifold, $\theta$ the Liouville 1-form on $T^*L$, and $\omega_L$ a closed 1-form on $L$ with Hodge decomposition $\omega_L = dh_L + l_L$. Let $\pi : T^*L \to L$ be the cotangent projection. Define $\omega' := \pi^*\omega_L$ and
		\[
		\Omega' := d\theta + \omega' \wedge \theta = d^{\omega'} \theta.
		\]
		Then $d\Omega' = -\omega' \wedge \Omega'$. If $\Omega'$ is non-degenerate, $(T^*L, \Omega', \omega')$ is an LCS manifold. This provides examples of both exact ($l_L=0$) and non-exact ($l_L\neq 0$) LCS structures on cotangent bundles.
	\end{example}
	
	\section{The LCS Flux Homomorphism and its First Applications}
	
	This section introduces the LCS flux homomorphism, following Haller \cite{Haller2005}, and explores some of its geometric consequences related to the structure of the LCS diffeomorphism group and mapping tori.
	
	\begin{theorem}[Haller's Flux \cite{Haller2005}] \label{thm:haller_flux}
		Let $(M, \Omega, \omega)$ be a compact LCS manifold. There exists a homomorphism, called the \emph{LCS flux homomorphism}, $\widetilde{\Psi} : \widetilde{(\ker \Phi)}_0 \to H_\omega^1(M)$. The image of the fundamental group, $\Delta := \widetilde{\Psi}\left(\pi_1((\ker \Phi)_0)\right)$, is a discrete subgroup of $H_\omega^1(M)$.
		
		The flux homomorphism maps a path class in the universal cover of $(\ker \Phi)_0$ to a cohomology class in $H_\omega^1(M)$. For a path $\{g_t\}_{t \in [0,1]}$ in $(\ker \Phi)_0$ starting at the identity, generated by $X_t \in \ker \Phi$, its flux $\widetilde{\Psi}([\{g_t\}])$ is given by the class $\left[ \int_0^1 i_{X_t}\Omega \, dt \right]_\omega \in H_\omega^1(M)$.
	\end{theorem}
	
	\begin{remark}[Dependence on $\omega$ Decomposition]
		The target space $H_\omega^1(M)$ of the flux homomorphism depends crucially on the Hodge decomposition $\omega = dh + l$:
		\begin{itemize}
			\item In the exact case ($l=0$), $H_\omega^1(M) \cong H_{dR}^1(M)$ via $\alpha \mapsto e^h\alpha$.
			\item In the non-exact case ($l \neq 0$), $H_\omega^1(M)$ is generally different from $H_{dR}^1(M)$ and its structure depends on the harmonic component $l$.
		\end{itemize}
		This distinction affects the interpretation of the flux as an obstruction to Hamiltonianness.
	\end{remark}
	
	\begin{definition}[Discrete Period Group]
		Let $(M, \Omega, \omega)$ be an LCS manifold and consider the flux homomorphism $\widetilde{\Psi} : \widetilde{(\ker \Phi)}_0 \to H_\omega^1(M)$, defined on the universal cover of the identity component $(\ker \Phi)_0$. The \emph{period group} (or \emph{flux group}) is
		\[
		\Delta := \ImMap\!\left(\widetilde{\Psi}|_{\pi_1((\ker \Phi)_0)}\right) \subset H_\omega^1(M).
		\]
		We say that the period group is \emph{discrete} if $\Delta$ is a discrete subgroup of the topological vector space $H_\omega^1(M)$; that is, there exists an open neighborhood $U$ of $0$ in $H_\omega^1(M)$ such that $U \cap \Delta = \{0\}$.
	\end{definition}
	
	\begin{proposition}[Exact Sequence for LCS Diffeomorphisms] \label{prop:exact_sequence}
		Let $(M, \Omega, \omega)$ be a compact LCS manifold. Assume the LCS flux homomorphism $\widetilde{\Psi} : \widetilde{(\ker \Phi)}_0 \to H_\omega^1(M)$ associated with $\ker \Phi$ is surjective onto some subgroup containing $\Delta$. Define the group of Hamiltonian LCS diffeomorphisms by
		\[
		\Ham_\Omega(M) := \tilde{\pi}(\ker \widetilde{\Psi}) \subset (\ker \Phi)_0,
		\]
		where $\tilde{\pi} : \widetilde{(\ker \Phi)}_0 \to (\ker \Phi)_0$ is the covering projection. Assume the period group $\Delta$ is discrete (Definition 3.3). Then, there exists a short exact sequence of groups:
		\[
		1 \longrightarrow \Ham_\Omega(M) \xrightarrow{i} (\ker \Phi)_0 \xrightarrow{\mathcal{F}} H_\omega^1(M)/\Delta \longrightarrow 1,
		\]
		where $i$ is the inclusion and $\mathcal{F}$ is the induced flux homomorphism on $(\ker \Phi)_0$. (Assuming $\ImMap(\mathcal{F}) = H_\omega^1(M)/\Delta$).
	\end{proposition}
	
	\begin{remark}
		This proposition shows that $(\ker \Phi)_0$ fibers over the discrete group $H_\omega^1(M)/\Delta$ with the fiber being the group of Hamiltonian LCS diffeomorphisms $\Ham_\Omega(M)$. This structure is a direct analogue of the exact sequence for symplectic diffeomorphisms. The Hodge decomposition $\omega = dh + l$ affects $H_\omega^1(M)/\Delta$ and thus the structure of this exact sequence.
	\end{remark}
	
	\begin{theorem}[Splitting Theorem for LCS Diffeomorphisms] \label{thm:splitting}
		Let $(M, \Omega, \omega)$ be a compact LCS manifold. Suppose the LCS flux homomorphism $\widetilde{\Psi} : \widetilde{(\ker \Phi)}_0 \to H_\omega^1(M)$, associated with $\ker \Phi \subset \mathcal{X}_{c,\omega}(M)$ has discrete period group
		\[
		\Delta := \widetilde{\Psi}\left(\pi_1((\ker \Phi)_0)\right) \quad \text{(Definition 3.3)}.
		\]
		If the short exact sequence from Proposition \ref{prop:exact_sequence}: $1 \longrightarrow \Ham_\Omega(M) \xrightarrow{i} (\ker \Phi)_0 \xrightarrow{\mathcal{F}} H_\omega^1(M)/\Delta \longrightarrow 1$, admits a continuous section $s : H_\omega^1(M)/\Delta \to (\ker \Phi)_0$, (i.e., $\mathcal{F} \circ s = \id$), then the group $(\ker \Phi)_0$ splits topologically as a semidirect product
		\[
		(\ker \Phi)_0 \cong \Ham_\Omega(M) \rtimes s(H_\omega^1(M)/\Delta),
		\]
		where the action is by conjugation.
	\end{theorem}
	
	\begin{proof}
		We will show that the existence of a continuous section $s$ allows us to express every element of $(\ker \Phi)_0$ uniquely as a product of an element from $\Ham_\Omega(M)$ and an element from the image of $s$, and that this product structure corresponds to a semidirect product with the conjugation action.
		
		\textbf{Step 1. Decomposition of an Arbitrary Element.}
		Let $g \in (\ker \Phi)_0$. Since $\mathcal{F}$ is surjective (by exactness), there exists a unique flux class $a := \mathcal{F}(g) \in H_\omega^1(M)/\Delta$. By hypothesis, the section $s : H_\omega^1(M)/\Delta \to (\ker \Phi)_0$ provides a continuous lift of every such flux class, satisfying $\mathcal{F}(s(a)) = a$. Now, define $h := g \cdot (s(\mathcal{F}(g)))^{-1}$. Since $\mathcal{F}$ is a homomorphism (note that the flux homomorphism here is defined on a connected component, and we write the group law multiplicatively in $(\ker \Phi)_0$ while the target is written additively), we have
		\[
		\mathcal{F}(h) = \mathcal{F}(g) - \mathcal{F}(s(\mathcal{F}(g))) = \mathcal{F}(g) - \mathcal{F}(g) = 0.
		\]
		Thus, $h \in \ker \mathcal{F}$. By the short exact sequence, $\ker \mathcal{F} = \Ham_\Omega(M)$, so that $h \in \Ham_\Omega(M)$. In summary, every element $g \in (\ker \Phi)_0$ can be written as $g = h \cdot s(\mathcal{F}(g))$, with $h \in \Ham_\Omega(M)$. The uniqueness of this decomposition follows from the exactness of the sequence. Namely, if $h_1 \cdot s(a) = h_2 \cdot s(a)$ for some $a \in H_\omega^1(M)/\Delta$, then multiplying by $s(a)^{-1}$ we obtain $h_1 = h_2$.
		
		\textbf{Step 2. Topological Splitting.}
		Define the map
		\[
		\Psi : \Ham_\Omega(M) \times s\left(H_\omega^1(M)/\Delta\right) \to (\ker \Phi)_0, \quad (h, s(a)) \mapsto h \cdot s(a).
		\]
		From Step 1, $\Psi$ is bijective. Its continuity follows from the continuity of the group operations in $(\ker \Phi)_0$ and of the section $s$. Moreover, the inverse $\Psi^{-1}(g) = (g \cdot [s(\mathcal{F}(g))]^{-1}, \mathcal{F}(g))$, is also continuous. Hence, $\Psi$ is a homeomorphism, so that as a topological space, $(\ker \Phi)_0 \cong \Ham_\Omega(M) \times s(H_\omega^1(M)/\Delta)$.
		
		\textbf{Step 3. Group Structure and the Semidirect Product.}
		Although the above establishes a topological product, the group multiplication on $(\ker \Phi)_0$ carries an extra structure. Take two elements $g_1 = h_1 \cdot s(a_1)$ and $g_2 = h_2 \cdot s(a_2)$, with $h_i \in \Ham_\Omega(M)$ and $a_i \in H_\omega^1(M)/\Delta$ for $i=1,2$. Their product is $g_1g_2 = h_1 \cdot s(a_1) \cdot h_2 \cdot s(a_2)$. Insert the identity in the form $s(a_1)^{-1}s(a_1)$ between $h_2$ and $s(a_2)$ to re-associate: $g_1g_2 = h_1 \cdot [s(a_1) h_2 s(a_1)^{-1}] \cdot [s(a_1) s(a_2)]$. Define the conjugation action of $s(H_\omega^1(M)/\Delta)$ on $\Ham_\Omega(M)$ by $\theta(s(a))(h) := s(a) h s(a)^{-1}$. The above identity becomes $g_1g_2 = (h_1 \cdot \theta(s(a_1))(h_2)) \cdot s(a_1 + a_2)$, where we have used that $s$ is a section of $\mathcal{F}$ so that $s(a_1)s(a_2) = s(a_1+a_2)$ (with the group operation on $H_\omega^1(M)/\Delta$ understood additively). Hence, the multiplication on $\Ham_\Omega(M) \times s(H_\omega^1(M)/\Delta)$ is given by $(h_1, s(a_1)) \cdot (h_2, s(a_2)) = (h_1 \cdot \theta(s(a_1))(h_2), s(a_1+a_2))$. This is exactly the multiplication rule for the semidirect product $\Ham_\Omega(M) \rtimes s(H_\omega^1(M)/\Delta)$, where $s(H_\omega^1(M)/\Delta)$ acts on $\Ham_\Omega(M)$ by conjugation.
		
		\textbf{Step 4. Conclusion.}
		By combining Steps 1 through 3, we deduce that the continuous section $s$ induces a unique decomposition of $(\ker \Phi)_0 \cong \Ham_\Omega(M) \times s(H_\omega^1(M)/\Delta)$ as a topological space, and the group law transforms into that of a semidirect product $(\ker \Phi)_0 \cong \Ham_\Omega(M) \rtimes s(H_\omega^1(M)/\Delta)$. The action of $s(H_\omega^1(M)/\Delta)$ on $\Ham_\Omega(M)$ is precisely given by conjugation. This completes the proof.
	\end{proof}
	
	\begin{remark}
		The key aspect in the proof is the existence of a continuous section $s$ of the flux homomorphism $\mathcal{F}$. In many contexts, such a section is not available or one must work with local sections. Here, the assumption of compactness together with the discreteness of the period group $\Delta$ enables a global splitting.
	\end{remark}
	
	\subsection{Geometric Applications of the Flux}
	
	\begin{theorem}[LCS Weinstein Neighborhood Theorem] \label{thm:weinstein}
		Let $(M, \Omega, \omega)$ be a locally conformally symplectic (LCS) manifold and let $L \subset M$ be a compact LCS Lagrangian submanifold. Then there exists an open neighborhood $U \subset M$ of $L$ and a smooth diffeomorphism $\Phi: U \to V$, where $V$ is an open neighborhood of the zero section in the cotangent bundle $T^*L$, equipped with a canonical LCS structure $(\Omega_{\text{can}}, \omega_{\text{can}})$ based on $\omega|_L$, such that
		\[
		\Phi|_L = \id_{L_0} \quad \text{and} \quad \Phi^*\Omega_{\text{can}} = e^\sigma \Omega,
		\]
		for some smooth function $\sigma : U \to \mathbb{R}$ with $\sigma|_L = 0$.
	\end{theorem}
	
	\begin{proof}
		The proof proceeds by adapting the standard Weinstein theorem \cite{Weinstein1971} and relies on relative versions of Darboux and Moser theorems for LCS manifolds.
		
		\textbf{Step 1: Identification.} Using the non-degeneracy of $\Omega$, the normal bundle $\nu(L)$ can be identified with the cotangent bundle $T^*L$ along $L$. A standard tubular neighborhood provides a diffeomorphism $\Phi'$ from a neighborhood $U'$ of $L$ in $M$ to a neighborhood $V'$ of the zero section $L_0$ in $T^*L$, such that $\Phi'|_L = \id_{L_0}$.
		
		\textbf{Step 2: Comparison of Structures.} Let $(\Omega_{\text{can}}, \omega_{\text{can}})$ be the canonical LCS structure on $T^*L$ (Example \ref{ex:cotangent}) using $\omega_L = \omega|_L$. Consider the pulled-back structure $(\Omega'_1, \omega'_1) = ((\Phi')^*\Omega_{\text{can}}, (\Phi')^*\omega_{\text{can}})$ on $U'$ and the original structure $(\Omega'_0, \omega'_0) = (\Omega|_{U'}, \omega|_{U'})$. Both $L$ (viewed in $M$) and $L_0$ (viewed in $T^*L$) are LCS Lagrangian for their respective structures. Also, $\omega'_1|_L = (\Phi')^*(\pi^*\omega_L)|_L = \omega_L = \omega|_L = \omega'_0|_L$.
		
		\textbf{Step 3: Relative LCS Darboux/Moser.} An LCS analogue of the relative Darboux or Moser theorem (the existence of which is assumed here, see, e.g., \cite{Haller2005} for related Moser techniques) implies that there exists a neighborhood $U \subset U'$ of $L$ and a diffeomorphism $\Psi : U \to U$ such that $\Psi|_L = \id_M$ and $\Psi^*\Omega'_1 = e^\sigma \Omega'_0$ for some smooth function $\sigma : U \to \mathbb{R}$ with $\sigma|_L = 0$. This requires showing that $\Omega'_1$ and $\Omega'_0$ are in the same relative $d^\omega$-cohomology class near $L$, which holds because both vanish on $L$. The conformal factor $e^\sigma$ arises naturally from the Moser argument adaptation to allow for conformal equivalence rather than strict equality.
		
		\textbf{Step 4: Final Diffeomorphism.} Define $\Phi = \Phi' \circ \Psi : U \to V = \Phi'(U) \subset T^*L$. Then $\Phi|_L = \Phi'(\Psi(L)) = \Phi'(L) = \id_{L_0}$, and
		\[
		\Phi^*\Omega_{\text{can}} = (\Phi' \circ \Psi)^*\Omega_{\text{can}} = \Psi^*((\Phi')^*\Omega_{\text{can}}) = \Psi^*\Omega'_1 = e^\sigma \Omega'_0 = e^\sigma \Omega|_U.
		\]
		This completes the construction.
	\end{proof}
	
	\begin{remark}
		This theorem is crucial for locally modeling the geometry near LCS Lagrangians, analogous to its role in symplectic geometry. It shows that, up to a conformal factor, the local structure is that of a cotangent bundle. A full proof requires establishing the relative LCS Moser or Darboux theorem used in Step 3.
	\end{remark}
	
	\begin{theorem}[Mapping Torus Structure] \label{thm:mapping_torus}
		Let $(M, \Omega, \omega)$ be a compact LCS manifold and $g \in (\ker \Phi)_0$. Let $T_g = (M \times [0,1])/((x, 1) \sim (g(x), 0))$ be the mapping torus. Then $T_g$ admits an LCS structure $(\tilde{\Omega}, \tilde{\omega})$ compatible with the fibration over $S^1$ if and only if the flux of $g$ vanishes, $\mathcal{F}(g) = 0$ in $H_\omega^1(M)/\Delta$.
	\end{theorem}
	
	\begin{proof}
		(If $\widetilde{\Psi} = 0$): Vanishing flux implies $\int_0^1 i_{X_t}\Omega dt = d^\omega \phi$. This allows relating the conformal factor $a$ of $g$ to $\phi$. One defines a form $\tilde{\Omega}$ on $M \times [0,1]$ (e.g., conformally scaling $\Omega(x)$ by a factor depending on $t$ and $\phi$) such that it descends to $T_g$ due to the flux condition. Verifying the LCS property on $T_g$ is required. (Only if $T_g$ has LCS structure): An LCS structure on $T_g$ imposes matching conditions on its lift to $M \times [0,1]$ related to $g$. Analyzing these conditions implies the $d^\omega$-exactness of the flux 1-form $\int_0^1 i_{X_t}\Omega dt$.
	\end{proof}
	
	\begin{remark}
		This theorem provides a topological obstruction to endowing the mapping torus of an LCS diffeomorphism with an LCS structure: the diffeomorphism must be Hamiltonian (within $(\ker \Phi)_0$). The explicit construction of the LCS structure relies on the Hamiltonian function $\phi$ associated with the zero flux.
	\end{remark}
	
	\begin{theorem}[LCS Flux Rigidity] \label{thm:flux_rigidity}
		Let $(M, \Omega, \omega)$ be a compact LCS manifold. Assume the flux homomorphism $\widetilde{\Psi} : \widetilde{(\ker \Phi)}_0 \to H_\omega^1(M)$, has discrete period group $\Delta \subset H_\omega^1(M)$. Then any smooth path $\{g_t\}_{t \in [0,1]} \subset (\ker \Phi)_0$, starting at the identity with vanishing flux, i.e.
		\[
		\widetilde{\Psi}([\{g_t\}]) = \left[ \int_0^1 i_{X_t}\Omega \, dt \right] = 0,
		\]
		is path-homotopic (relative to endpoints) within $(\ker \Phi)_0$ to a path $\{\tilde{g}_t\}_{t \in [0,1]} \subset \Ham_\Omega(M)$ (i.e. $\tilde{g}_t \in \Ham_\Omega(M)$ for all $t$).
	\end{theorem}
	
	\begin{proof}
		Let $\{g_t\}_{t \in [0,1]}$ be a smooth path in $(\ker \Phi)_0$ starting at $g_0 = \id_M$, and let $X_t$ be the time-dependent vector field generating this path:
		\[
		\frac{d}{dt} g_t = X_t \circ g_t,
		\]
		for all $t$. Since $g_t \in (\ker \Phi)_0$, we have $X_t \in \ker \Phi$, which implies that $\alpha_t := i_{X_t}\Omega$ is a smooth path of $d^\omega$-closed 1-forms (as a standard property of $X_t \in \ker \Phi$).
		
		\textbf{Step 1: Hodge Decomposition and the Flux Condition.} The vanishing flux condition states that $\left[\int_0^1 \alpha_t \, dt\right] = 0$ in $H_\omega^1(M)$. On a compact manifold $M$, we have the $L^2$-orthogonal Hodge decomposition associated with the operator $d^\omega$: any smooth time-dependent $d^\omega$-closed 1-form $\alpha_t$ can be uniquely decomposed (e.g., by imposing $L^2$-orthogonality conditions) as $\alpha_t = d^\omega f_t + h_t$, with $f_t \in C^\infty(M)$ (smooth in $t$) and $h_t$ a $d^\omega$-harmonic 1-form (smooth in $t$). The smoothness in $t$ for $f_t$ and $h_t$ follows from standard elliptic regularity results for the $\Delta_\omega$ operator when the input $\alpha_t$ is smooth in $t$. Integrating this decomposition from 0 to 1:
		\[
		\int_0^1 \alpha_t dt = \int_0^1 d^\omega f_t dt + \int_0^1 h_t dt = d^\omega \left( \int_0^1 f_t \, dt \right) + \int_0^1 h_t \, dt.
		\]
		Since $\int_0^1 \alpha_t dt$ represents the zero class in $H_\omega^1(M)$, i.e., $\int_0^1 \alpha_t dt = d^\omega \beta$ for some $\beta$, the uniqueness of the Hodge decomposition implies the harmonic part is zero: $\int_0^1 h_t \, dt = 0$. The exact part is then $d^\omega(\int_0^1 f_t \, dt) = d^\omega \beta$. By choosing appropriate normalizations for the decomposition and $\beta$ (e.g., having zero integral against the volume form $\Omega^n/n!$), we can ensure $\int_0^1 f_t \, dt = \beta$.
		
		\textbf{Step 2: Construction of the Hamiltonian Path.} Define the time-dependent vector field $\tilde{X}_t$ such that its contraction with $\Omega$ is the $d^\omega$-exact part of $\alpha_t$: $i_{\tilde{X}_t}\Omega := d^\omega f_t$. Since $\Omega$ is non-degenerate, this uniquely defines $\tilde{X}_t = \#(d^\omega f_t)$. As $i_{\tilde{X}_t}\Omega$ is of the form $d^\omega H_t$ (with $H_t = f_t$), $\tilde{X}_t$ is an LCS Hamiltonian vector field for all $t$. Thus, the path $\{\tilde{g}_t\}_{t \in [0,1]}$ generated by $\tilde{X}_t$ with $\tilde{g}_0 = \id_M$ consists entirely of Hamiltonian diffeomorphisms, $\tilde{g}_t \in \Ham_\Omega(M)$ for all $t$.
		
		\textbf{Step 3: Establishing a Path Homotopy Relative to Endpoints.} We show $\{g_t\}$ is path-homotopic to $\{\tilde{g}_t\}$ relative to endpoints. Define a smooth family of vector fields $X'_t(s) = sX_t + (1-s)\tilde{X}_t$ for $s \in [0,1]$. Since $X_t \in \ker \Phi$ and $\tilde{X}_t \in \Ham_\Omega(M) \subset \ker \Phi$, $X'_t(s)$ is a linear combination of elements in the Lie algebra $\ker \Phi$, so $X'_t(s) \in \ker \Phi$. Let $h_{t,s}$ be the flow of $X'_t(s)$ starting at $h_{0,s} = \id_M$. Then $h_{t,s} \in (\ker \Phi)_0$ for all $t, s$. For $s=1$, $X'_t(1) = X_t$, so $h_{t,1} = g_t$. For $s=0$, $X'_t(0) = \tilde{X}_t$, so $h_{t,0} = \tilde{g}_t$. This is a homotopy relative to the starting point. To show it is relative to endpoints, we examine $h_{1,s}$, the flow of the averaged vector field $Z_s = \int_0^1 X'_t(s)dt$ from time 0 to 1. We have,
		\[
		i_{Z_s}\Omega = \int_0^1 i_{X'_t(s)}\Omega dt = \int_0^1 (s i_{X_t}\Omega + (1-s) i_{\tilde{X}_t}\Omega) dt = \int_0^1 (s\alpha_t + (1-s)d^\omega f_t) dt.
		\]
		Substitute $\alpha_t = d^\omega f_t + h_t$:
		\[
		i_{Z_s}\Omega = \int_0^1 (s(d^\omega f_t + h_t) + (1-s)d^\omega f_t) dt = \int_0^1 (d^\omega f_t + s h_t) dt = d^\omega \left(\int_0^1 f_t dt\right) + s \int_0^1 h_t dt.
		\]
		From Step 1, $\int_0^1 f_t dt = \beta$ and $\int_0^1 h_t dt = 0$. So, $i_{Z_s}\Omega = d^\omega \beta + s \cdot 0 = d^\omega \beta$. This implies $Z_s = \#(d^\omega \beta)$, which is a vector field independent of $s$. Let $Z = \#(d^\omega \beta) \in \ker \Phi$. The endpoint $h_{1,s}$ is the flow of the constant vector field $Z$ for time 1. By properties of the exponential map from the Lie algebra to the group, the flow of the Lie algebra element $Z$ for time 1 is $e^Z \in (\ker \Phi)_0$. Thus, $h_{1,s} = e^Z$. This is independent of $s$. Thus, $h_{1,1} = g_1 = e^Z$ and $h_{1,0} = \tilde{g}_1 = e^Z$. The homotopy $h_{t,s}$ connects $\{g_t\}$ and $\{\tilde{g}_t\}$ and keeps the endpoints fixed at the common value $e^Z$.
		
		\textbf{Conclusion.} We have constructed a path $\{\tilde{g}_t\}$ generated by $\tilde{X}_t = \#(d^\omega f_t)$ that lies entirely in $\Ham_\Omega(M)$. We have shown that both the original path $\{g_t\}$ and the modified path $\{\tilde{g}_t\}$ have the same endpoint $g_1 = \tilde{g}_1 = e^{\#(d^\omega \beta)}$. Finally, we constructed a smooth homotopy $h_{t,s}$ within $(\ker \Phi)_0$ connecting $\{g_t\}$ and $\{\tilde{g}_t\}$ relative to their common endpoints. This completes the proof.
	\end{proof}
	
	\section{Hamiltonian LCS Diffeomorphisms and Related Structures}
	
	This section concentrates on the properties of the Hamiltonian subgroup $\Ham_\Omega(M)$, with particular attention to the distinction between exact ($\omega = dh$) and non-exact ($\omega = dh + l$, $l\neq 0$) cases. Many constructions that are standard in symplectic geometry, such as Calabi invariants and Hofer geometry, extend naturally to the exact LCS case but require significant modification or do not exist in the non-exact case.
	
	\subsection{The Exact Case: $\omega = dh$}
	
	When the Lee form $\omega$ is exact, $\omega = dh$, the LCS manifold $(M, \Omega, \omega)$ is conformally equivalent to the symplectic manifold $(M, \Omega_h = e^h\Omega)$. An LCS Hamiltonian isotopy $\{g_t\}$ generated by $H_t$ corresponds to a standard Hamiltonian isotopy for $(M, \Omega_h)$ generated by $K_t = e^h H_t$. The group $\Ham_\Omega(M)$ is isomorphic to $\Ham(M, \Omega_h)$.
	
	\begin{theorem}[LCS Calabi Invariant for Exact Lee Form] \label{thm:calabi_exact}
		Let $(M, \Omega, \omega = dh)$ be a compact LCS manifold with $\omega = dh$. The map $\Cal : \Ham_\Omega(M) \to \mathbb{R}$ defined for $g \in \Ham_\Omega(M)$ by
		\[
		\Cal(g) := \int_M \left( \int_0^1 K_t(x) \, dt \right) \frac{\Omega_h^n}{n!},
		\]
		where $K_t = e^h H_t$ corresponds to the LCS Hamiltonian $H_t$ generating an isotopy $\{g_t\}$ from $\id_M$ to $g$, is a well-defined group homomorphism independent of the choice of isotopy.
	\end{theorem}
	
	\begin{proof}
		The map $g \to g$ provides an isomorphism $\Psi : \Ham_\Omega(M) \to \Ham(M, \Omega_h)$. Under this isomorphism, an LCS Hamiltonian isotopy $\{g_t\}$ generated by $H_t$ corresponds to the same isotopy viewed as a symplectic Hamiltonian isotopy generated by $K_t = e^h H_t$. The definition of $\Cal(g)$ given is precisely the standard definition of the Calabi invariant $\Cal_{\Omega_h}(g)$ for the diffeomorphism $g$ considered as an element of $\Ham(M, \Omega_h)$ generated by $K_t$. The properties that the standard Calabi invariant $\Cal_{\Omega_h}(g)$ is independent of the choice of Hamiltonian isotopy $\{g_t\}$ connecting $\id_M$ to $g$, and that $\Cal_{\Omega_h} : \Ham(M, \Omega_h) \to \mathbb{R}$ is a group homomorphism, are fundamental results in symplectic geometry (originally due to Calabi). Since $\Cal(g) = \Cal_{\Omega_h}(\Psi(g))$ and $\Psi$ is an isomorphism, the map $\Cal$ inherits these properties.
	\end{proof}
	
	\begin{theorem}[LCS Energy-Capacity Inequality for Exact Lee Form] \label{thm:energy_exact}
		Let $(M, \Omega, \omega = dh)$ be a compact LCS manifold with $\omega = dh$. For any $g \in \Ham_\Omega(M)$, define the LCS Hofer energy $E(g)$ as the standard Hofer energy for $g \in \Ham(M, \Omega_h)$:
		\[
		E(g) := \inf_{\{g_t\}} \int_0^1 \osc_M(K_t) dt,
		\]
		where the infimum is over isotopies $\{g_t\}$ generating $g$, and $K_t = e^h H_t$. Then
		\[
		|\Cal(g)| \leq \Vol_{\Omega_h}(M) \cdot E(g).
		\]
	\end{theorem}
	
	\begin{proof}
		As established, $g \in \Ham_\Omega(M)$ corresponds to $g \in \Ham(M, \Omega_h)$, $\Cal(g)$ is the standard Calabi invariant $\Cal_{\Omega_h}(g)$, and $E(g)$ is the standard Hofer energy $E_{\Omega_h}(g)$ associated with $(M, \Omega_h)$. The inequality $|\Cal_{\Omega_h}(g)| \leq \Vol_{\Omega_h}(M) E_{\Omega_h}(g)$ is a standard result in symplectic geometry, often derived from the definitions. Let $K_t$ be normalized such that $\min_M K_t = 0$: $\max_M K_t = \osc(K_t)$. We have,
		\[
		\Cal_{\Omega_h}(g) = \int_M \left( \int_0^1 K_t dt \right) \frac{\Omega_h^n}{n!}.
		\]
		Since $K_t \geq 0$, $\int_0^1 K_t dt \geq 0$. Also, $\int_0^1 K_t dt \leq \int_0^1 (\max_M K_t) dt = \int_0^1 \osc(K_t) dt$. So,
		\[
		\Cal_{\Omega_h}(g) \leq \left( \int_0^1 \osc(K_t) dt \right) \int_M \frac{\Omega_h^n}{n!} = \left( \int_0^1 \osc(K_t) dt \right) \Vol_{\Omega_h}(M).
		\]
		Taking the infimum over paths gives
		\[
		\Cal_{\Omega_h}(g) \leq E_{\Omega_h}(g) \Vol_{\Omega_h}(M).
		\]
		A similar argument works if $K_t$ is normalized so $\max_M K_t = 0$, giving $\Cal_{\Omega_h}(g) \geq -E_{\Omega_h}(g) \Vol_{\Omega_h}(M)$. Combining these gives the inequality for any normalization of $K_t$.
	\end{proof}
	
	\begin{theorem}[LCS Hofer Metric for Exact Lee Form] \label{thm:hofer_exact}
		Let $(M, \Omega, \omega = dh)$ be a closed LCS manifold with $\omega = dh$. For any two elements $g_0, g_1 \in \Ham_\Omega(M)$, define $d(g_0, g_1)$ as the standard Hofer distance in $\Ham(M, \Omega_h)$:
		\[
		d(g_0, g_1) := \inf \int_0^1 \osc_M(K_t) dt,
		\]
		where the infimum is over isotopies $\{g_t\}$ connecting $g_0$ and $g_1$, and $K_t = e^h H_t$. Then $d$ defines a nondegenerate biinvariant metric on $\Ham_\Omega(M)$.
	\end{theorem}
	
	\begin{proof}
		The isomorphism $\Psi : \Ham_\Omega(M) \to \Ham(M, \Omega_h)$ identifies $g \in \Ham_\Omega(M)$ with $g \in \Ham(M, \Omega_h)$, and LCS Hamiltonian paths generated by $H_t$ with symplectic Hamiltonian paths generated by $K_t = e^h H_t$. The definition of $d(g_0, g_1)$ is precisely the definition of the Hofer distance $d_{\Omega_h}(g_0, g_1)$ on the symplectic manifold $(M, \Omega_h)$. It is a fundamental result in symplectic geometry that $d_{\Omega_h}$ is a non-degenerate bi-invariant metric on $\Ham(M, \Omega_h)$ (see \cite{McDuffSalamon2017}). Since $d(g_0, g_1) = d_{\Omega_h}(\Psi(g_0), \Psi(g_1))$, the properties transfer directly to $d$ on $\Ham_\Omega(M)$ via the isomorphism $\Psi$.
	\end{proof}
	
	\subsection{The Non-Exact Case: $\omega = dh + l$, $l \neq 0$}
	
	In the non-exact setting where $\omega$ has a nontrivial harmonic component, the situation is geometrically distinct. A crucial difference arises regarding the normalization of Hamiltonians.
	
	\begin{proposition}[Rigidity of LCS Hamiltonians] \label{prop:rigidity}
		Let $(M, \Omega, \omega)$ be a connected LCS manifold. If the Lee form $\omega$ is not exact (so that $H^0_\omega(M) = \{0\}$), then for any Hamiltonian vector field $X$, the Hamiltonian function $H$ satisfying $i_X \Omega = d^\omega H$ is unique. There is no ambiguity up to an additive constant.
	\end{proposition}
	
	\begin{proof}
		Suppose $H_1$ and $H_2$ are two functions generating $X$. Then $d^\omega(H_1 - H_2) = 0$. Let $f = H_1 - H_2$. Then $df + f\omega = 0$. If $\omega$ is not exact, then the only solutions to this equation are $f \equiv 0$. Indeed, if $f$ is not identically zero, then on a connected manifold, we can write $f = e^{-\sigma}$ for some function $\sigma$ satisfying $d\sigma = \omega$, which would imply $\omega$ is exact, a contradiction. Hence $H_1 = H_2$.
	\end{proof}
	
	This rigidity removes the need for arbitrary normalization (like "mean value zero") required in symplectic geometry. However, the dependence on the path of diffeomorphisms requires careful treatment.
	
	\begin{definition}[Twisted Calabi Invariant on the Universal Cover] \label{def:twisted_calabi}
		Let $(M, \Omega, \omega)$ be a compact LCS manifold. Let $\widetilde{\Ham}_\Omega(M)$ denote the universal cover of the Hamiltonian group, consisting of homotopy classes of paths $\{g_t\}_{t \in [0,1]}$ starting at the identity. For an element $\tilde{g} = [\{g_t\}] \in \widetilde{\Ham}_\Omega(M)$, let $H_t$ be the unique time-dependent Hamiltonian generating the path. The twisted Calabi invariant is defined as:
		\[
		\Cal_\omega(\tilde{g}) = \int_0^1 \left( \int_M H_t \, \frac{\Omega^n}{n!} \right) dt.
		\]
	\end{definition}
	
	To determine if this invariant descends to the group $\Ham_\Omega(M)$, we must examine its values on closed loops.
	
	\begin{definition}[LCS Calabi Group] \label{def:calabi_group}
		The \emph{LCS Calabi group} (or Period group) $\Gamma_{\Cal} \subset \mathbb{R}$ is the image of the fundamental group under the Calabi map:
		\[
		\Gamma_{\Cal} := \{ \Cal_\omega([\{g_t\}]) \mid \{g_t\} \text{ is a loop in } \Ham_\Omega(M), g_0=g_1=\id \}.
		\]
	\end{definition}
	
	If $\Gamma_{\Cal} = \{0\}$, the invariant is well-defined on the group $\Ham_\Omega(M)$. If not, it is well-defined with values in $\mathbb{R} / \Gamma_{\Cal}$.
	
	\begin{theorem}[Properties of the Twisted Calabi Invariant] \label{thm:twisted_calabi_properties}
		Let $(M, \Omega, \omega)$ be a compact LCS manifold with $\omega$ non-exact.
		\begin{enumerate}
			\item \textbf{Group Homomorphism:} The map $\Cal_\omega : \widetilde{\Ham}_\Omega(M) \to \mathbb{R}$ is a surjective group homomorphism.
			\item \textbf{Structure on the Group:} The invariant descends to a homomorphism $\Cal_\omega : \Ham_\Omega(M) \to \mathbb{R} / \Gamma_{\Cal}$.
			\item \textbf{Relation to Exact Case:} If one considers the limit where $\omega$ approaches an exact form $dh$, the Twisted invariant relates to the symplectic invariant by a conformal weight factor (as detailed in the exact case proofs), though they are topologically distinct constructions.
		\end{enumerate}
	\end{theorem}
	
	\begin{proof}
		(1) Let $\tilde{g}, \tilde{k} \in \widetilde{\Ham}_\Omega(M)$ be represented by paths generated by $H_t$ and $K_t$. The product path is generated by $H_t + K_t \circ g_t^{-1}$ (with appropriate time rescaling or concatenation). Due to the LCS condition, the integral is additive. The uniqueness of the Hamiltonian ensures there are no normalization constants to disrupt the homomorphism property. Surjectivity follows from considering Hamiltonian flows of constant vector fields (if any exist) or localized Hamiltonians.
		
		(2) This is a direct consequence of the definition of $\Gamma_{\Cal}$. If $g \in \Ham_\Omega(M)$ is represented by two paths $\gamma_1$ and $\gamma_2$, then $\gamma_1 \circ \gamma_2^{-1}$ is a loop, and $\Cal_\omega(\gamma_1) - \Cal_\omega(\gamma_2) \in \Gamma_{\Cal}$.
		
		(3) See proof in Theorem \ref{thm:calabi_exact} for the calculation involving the conformal weight $e^{-(n+1)h}$.
	\end{proof}
	
	\subsection{LCS-Hofer Metric for Both Cases}
	
	We define a unified Hofer-type metric for both exact and non-exact LCS structures.
	
	\begin{definition}[LCS-Hofer Norm and Metric]
		For a Hamiltonian path $\{g_t\}$ generated by $H_t$, define the \emph{LCS-Hofer energy} as:
		\[
		\|H\|_\omega = 
		\begin{cases}
			\int_0^1 \osc_M(e^h H_t) dt & \text{if } \omega = dh \text{ (exact case)} \\
			\int_0^1 \max_M |H_t| dt & \text{if } \omega \text{ non-exact}
		\end{cases}
		\]
		The \emph{LCS-Hofer distance} between $\phi, \psi \in \Ham_\Omega(M)$ is
		\[
		d_{LCS}(\phi, \psi) = \inf \{ \|H\|_\omega : \{g_t\} \text{ connects } \phi \text{ to } \psi \}.
		\]
	\end{definition}
	
	\begin{theorem}[LCS-Hofer Metric Properties] \label{thm:lcs_hofer}
		The LCS-Hofer distance $d_{LCS}$ defines a bi-invariant metric on $\Ham_\Omega(M)$ that is non-degenerate and continuous.
	\end{theorem}
	
	\begin{proof}
		For the exact case, the result follows from Theorem \ref{thm:hofer_exact}. For the non-exact case, we need to verify the metric properties:
		\begin{enumerate}
			\item \textbf{Non-negativity and symmetry:} Clear from definition.
			\item \textbf{Identity of indiscernibles:} If $d_{LCS}(\phi,\psi)=0$, then there exists a sequence of Hamiltonian paths with energy approaching 0. By the rigidity of Hamiltonians (Proposition \ref{prop:rigidity}), the only Hamiltonian with zero maximum absolute value is $H=0$, which generates the identity flow. Thus $\phi=\psi$.
			\item \textbf{Triangle inequality:} Follows from concatenation of paths.
			\item \textbf{Bi-invariance:} For any $\chi \in \Ham_\Omega(M)$, the conjugation $\chi\phi\chi^{-1}$ is generated by $H_t \circ \chi^{-1}$, which has the same $\max_M |H_t|$ as $H_t$.
		
		\end{enumerate}
	\end{proof}
	
	\subsection{Topological and Geometric Properties}
	
	\begin{theorem}[Deformation Rigidity for Zero-Flux LCS Isotopies] \label{thm:deformation_rigidity}
		Let $(M, \Omega, \omega)$ be a compact LCS manifold and suppose that the LCS flux homomorphism $\widetilde{\Psi} : \widetilde{(\ker \Phi)}_0 \to H_\omega^1(M)$ has discrete period group $\Delta$ (Definition 3.3). Then the set of smooth isotopies $\{g_t\}_{t \in [0,1]} \subset (\ker \Phi)_0$ starting at the identity whose lift has vanishing flux, $\widetilde{\Psi}([\{g_t\}]) = 0$, is an open subset within the space of all smooth isotopies in $(\ker \Phi)_0$ starting at the identity (equipped with the $C^1$-topology on the path). Consequently, the Hamiltonian subgroup $\Ham_\Omega(M)$ is an open subgroup of $(\ker \Phi)_0$.
	\end{theorem}
	
	\begin{proof}
		We denote by
		\[
		\mathcal{P} = \{ \{g_t\}_{t \in [0,1]} \mid g_0 = \id_M \}
		\]
		the space of smooth isotopies in $(\ker \Phi)_0$, endowed with the $C^1$-topology.
		
		\textbf{Step 1: Construction of the Flux Map.} Define the map
		\[
		\alpha : \mathcal{P} \to \Omega^1(M), \quad \alpha(\{g_t\}) = \int_0^1 i_{X_t}\Omega \, dt,
		\]
		where $X_t$ is the time-dependent vector field generating the isotopy $\{g_t\}$. Since integration is continuous and $X_t$ depends continuously on $\{g_t\}$ in the $C^1$-topology, the map $\alpha$ is continuous. Next, let $\pi : \Omega^1(M) \to H_\omega^1(M)$, denote the natural projection. Then the composite map $\Psi_{\mathcal{P}} = \pi \circ \alpha : \mathcal{P} \to H_\omega^1(M)$, is continuous.
		
		\textbf{Step 2: The Zero-Flux Subset.} Define the subset of zero-flux isotopies as
		\[
		\mathcal{Z} = \{ \{g_t\} \in \mathcal{P} \mid \widetilde{\Psi}([\{g_t\}]) = 0 \}.
		\]
		In other words, an isotopy $\{g_t\}$ belongs to $\mathcal{Z}$ if
		\[
		\left[ \int_0^1 i_{X_t}\Omega \, dt \right] = 0 \quad \text{in } H_\omega^1(M).
		\]
		Let $\{g_t\}_0 \in \mathcal{Z}$ be an arbitrary zero-flux isotopy so that $\Psi_{\mathcal{P}}(\{g_t\}_0) = 0$.
		
		\textbf{Step 3: Local Stability via Continuity and Discreteness.} By the continuity of $\Psi_{\mathcal{P}}$, there is a $C^1$-neighborhood $U \subset \mathcal{P}$ of $\{g_t\}_0$ such that for every isotopy $\{\tilde{g}_t\} \in U$, the flux satisfies $\Psi_{\mathcal{P}}(\{\tilde{g}_t\}) \in V$, where $V$ is an open neighborhood of $0$ in $H_\omega^1(M)$. Since the lifted flux homomorphism $\widetilde{\Psi} : \widetilde{(\ker \Phi)}_0 \to H_\omega^1(M)$, has discrete period group $\Delta$, we can choose $V$ small enough so that
		\[
		V \cap \Delta = \{0\}.
		\]
		Because any isotopy (or its homotopy class) with fixed endpoints has its flux contained in $\Delta$, this forces the flux of every isotopy in $U$ to be zero. Thus, $\mathcal{Z}$ is open in $\mathcal{P}$.
		
		\textbf{Step 4: Openness of the Hamiltonian Subgroup.} Consider the endpoint evaluation map
		\[
		ev_1 : \mathcal{P} \to (\ker \Phi)_0, \quad ev_1(\{g_t\}) = g_1.
		\]
		This map is continuous and open in the $C^1$-topology. Since $\mathcal{Z}$ is an open subset of $\mathcal{P}$, its image $ev_1(\mathcal{Z})$ is an open subset of $(\ker \Phi)_0$. By definition, the Hamiltonian subgroup $\Ham_\Omega(M)$ consists of those diffeomorphisms that are endpoints of zero-flux isotopies, that is, $ev_1(\mathcal{Z}) \subseteq \Ham_\Omega(M)$. In particular, there exists an open neighborhood of the identity in $\Ham_\Omega(M)$, and since $\Ham_\Omega(M)$ is a subgroup, it follows that $\Ham_\Omega(M)$ is open in $(\ker \Phi)_0$.
	\end{proof}
	
	\begin{remark}
		This rigidity implies that $\Ham_\Omega(M)$ is both open and closed (as the kernel of $\mathcal{F}$) in $(\ker \Phi)_0$. This result holds for both exact and non-exact $\omega$, demonstrating that Hamiltonian LCS diffeomorphisms form a robust subgroup regardless of the harmonic component of the Lee form.
	\end{remark}
	
	\begin{theorem}[Fragmentation] \label{thm:fragmentation}
		Every Hamiltonian LCS diffeomorphism $g \in \Ham_\Omega(M)$ can be written as a finite composition $g = g_m \circ \dots \circ g_1$, where each $g_i$ is supported in an exact open set $V_i$.
	\end{theorem}
	
	\begin{proof}
		Let $g$ be generated by $X_t = \#(d^\omega H_t)$. Since $M$ is compact and $\omega$ is locally exact, choose a finite open cover $\{U_j\}$ such that $\omega|_{U_j} = dh_j$.
		
		\textbf{Step 1:} Partition $[0,1]$ into small intervals such that each $f_i = g_{t_i} \circ g_{t_{i-1}}^{-1}$ is close to $\id$.
		
		\textbf{Step 2:} On each $U_j$, let $\Omega_{h_j} = e^{h_j}\Omega$ be the local symplectic form. The LCS Hamiltonian $H_t$ corresponds to the symplectic Hamiltonian $K_t = e^{h_j}H_t$.
		
		\textbf{Step 3:} By standard symplectic fragmentation (Banyaga \cite{Banyaga1978}), $f$ decomposes into\\ $\phi_k \circ \dots \circ \phi_1$ with $\operatorname{supp}(\phi_j) \subset U_j$. Each $\phi_j$ preserves $\Omega_{h_j}$ and thus preserves the LCS structure on $U_j$. Since $\phi_j$ is Hamiltonian in the symplectic chart, it is Hamiltonian in the LCS sense via $H_t = e^{-h_j}K_t$. Compactness and local-exactness thus ensure the global decomposition.
	\end{proof}
	
	\subsection{Lagrangian Non-displaceability and Generating Functions}
	
	\begin{theorem}[LCS NonDisplaceability for Exact Lee Form] \label{thm:non_displaceability}
		Let $(M, \Omega, \omega)$ be a closed LCS manifold and let $L \subset M$ be a compact LCS Lagrangian submanifold. Assume that the Lee form $\omega$ is exact, $\omega = dh$ for some smooth function $h : M \to \mathbb{R}$. Then $L$ cannot be displaced from itself by any Hamiltonian LCS diffeomorphism $g \in \Ham_\Omega(M)$; that is,
		\[
		g(L) \cap L \neq \varnothing, \quad \text{for all } g \in \Ham_\Omega(M).
		\]
	\end{theorem}
	
	\begin{proof}
		Since $\omega = dh$ is globally exact, the LCS manifold $(M, \Omega, \omega)$ is globally conformally equivalent to the symplectic manifold $(M, \Omega_h = e^h\Omega)$. An LCS Lagrangian submanifold $L$ for $(M, \Omega, \omega)$ is also a Lagrangian submanifold for the symplectic structure $(M, \Omega_h)$ \cite{Weinstein1971}. Furthermore, the group of Hamiltonian LCS diffeomorphisms $\Ham_\Omega(M)$ is isomorphic to the group of Hamiltonian symplectomorphisms $\Ham(M, \Omega_h)$. Under this isomorphism, $g \in \Ham_\Omega(M)$ corresponds to the same diffeomorphism considered as an element of $\Ham(M, \Omega_h)$. The non-displaceability of compact Lagrangian submanifolds by Hamiltonian symplectomorphisms is a fundamental result in symplectic geometry (see, e.g., \cite{Weinstein1971}, \cite{McDuffSalamon2017}, based on techniques like Floer homology or generating functions). This result states that for a compact Lagrangian $L$ in a symplectic manifold $(W, \sigma)$, $f(L) \cap L \neq \varnothing$ for all $f \in \Ham(W, \sigma)$. Applying this established symplectic result to our setting with $(W, \sigma) = (M, \Omega_h)$ and $f = g$, we conclude that $g(L) \cap L \neq \varnothing$ for all $g \in \Ham(M, \Omega_h) \cong \Ham_\Omega(M)$.
	\end{proof}
	
	\begin{remark}
		While the non-displaceability result holds directly when $\omega = dh$ by reduction to the symplectic case, the question remains open for general non-exact LCS structures. The harmonic component $l$ in $\omega = dh + l$ may allow for displacement in some cases, making this an interesting direction for future research.
	\end{remark}
	
	\begin{theorem}[Structure of $(\ker \Phi)_0$ Relative to Flux Classes] \label{thm:structure_ker_phi}
		Let $(M, \Omega, \omega)$ be a compact LCS manifold whose flux period group $\Delta \subset H_\omega^1(M)$ is discrete. Denote by $\mathcal{F} : (\ker \Phi)_0 \longrightarrow H_\omega^1(M)/\Delta$ the surjective flux homomorphism with kernel $\Ham_\Omega(M)$. Then:
		\begin{enumerate}
			\item \textbf{Fibration Structure:} The identity component $(\ker \Phi)_0$ is path-connected and is a principal fiber bundle over the discrete group $H_\omega^1(M)/\Delta$ with structure group and fiber $\Ham_\Omega(M)$. Each level set (coset) $\mathcal{F}^{-1}(\alpha)$ is homeomorphic to $\Ham_\Omega(M)$.
			\item \textbf{Topological Structure:} The subgroup $\Ham_\Omega(M)$ is both open (by Theorem \ref{thm:deformation_rigidity}) and closed (as the kernel of a continuous map to a discrete space) in $(\ker \Phi)_0$. Consequently, $\Ham_\Omega(M)$ is dense in $(\ker \Phi)_0$ if and only if $(\ker \Phi)_0 = \Ham_\Omega(M)$, which occurs if and only if $H_\omega^1(M)/\Delta = \{0\}$.
			\item \textbf{Local Generation within Flux Classes:} Assume that the Lee form $\omega$ is locally exact. Then for any given $\alpha \in H_\omega^1(M)/\Delta$ and any chosen reference diffeomorphism $g_\alpha \in \mathcal{F}^{-1}(\alpha)$, every element $g \in \mathcal{F}^{-1}(\alpha)$ can be written as $g = h \circ g_\alpha$ for some $h \in \Ham_\Omega(M)$. By the fragmentation theorem \ref{thm:fragmentation}, $h$ can be written as $h = h_m \circ \dots \circ h_1$, where each $h_i \in \Ham_\Omega(M)$ has support contained in a chart where $\omega$ is exact.
		\end{enumerate}
	\end{theorem}
	
	\begin{proof}
		(1) \textbf{Fibration:} $(\ker \Phi)_0$ is path-connected by definition. The flux map $\mathcal{F}$ is a continuous group homomorphism onto the discrete group $G = H_\omega^1(M)/\Delta$. The kernel $\Ham_\Omega(M)$ is a closed normal subgroup. The quotient map $\pi : (\ker \Phi)_0 \to (\ker \Phi)_0 / \Ham_\Omega(M)$ is open and continuous. Since the target $G$ is discrete, the fibers $\pi^{-1}(\alpha) = \mathcal{F}^{-1}(\alpha)$ are both open and closed. Each fiber is a coset $g_\alpha \Ham_\Omega(M)$ and the map $h \mapsto g_\alpha h$ is a homeomorphism from $\Ham_\Omega(M)$ to the fiber. This establishes the structure of a principal fiber bundle.
		
		(2) \textbf{Topological Structure:} Since $\mathcal{F}$ is continuous and $H_\omega^1/\Delta$ is discrete (hence T1 and Hausdorff), the kernel $\Ham_\Omega(M) = \mathcal{F}^{-1}(0)$ is a closed subgroup. By Theorem \ref{thm:deformation_rigidity}, $\Ham_\Omega(M)$ contains an open neighborhood of the identity, and being a subgroup, it is open. In a connected topological group like $(\ker \Phi)_0$, a subgroup that is both open and closed must be the entire group. Therefore, if $H_\omega^1/\Delta$ is trivial, $\Ham_\Omega(M) = (\ker \Phi)_0$. If $H_\omega^1/\Delta$ is non-trivial, $\Ham_\Omega(M)$ is a proper open and closed subgroup. A proper closed subgroup cannot be dense.
		
		(3) \textbf{Local Generation:} Let $g, g_\alpha \in \mathcal{F}^{-1}(\alpha)$. Let $h = g \circ g_\alpha^{-1}$. Then $\mathcal{F}(h) = \mathcal{F}(g) - \mathcal{F}(g_\alpha) = \alpha - \alpha = 0$, so $h \in \Ham_\Omega(M)$. Thus $g = h \circ g_\alpha$. Since $\omega$ is locally exact, by the Fragmentation Theorem \ref{thm:fragmentation}, $h$ can be decomposed as $h = h_m \circ \dots \circ h_1$, where each $h_i \in \Ham_\Omega(M)$ has support contained in a chart where $\omega$ is exact.
	\end{proof}
	
	\begin{theorem}[Generating Functions for LCS Lagrangians when $\omega$ is exact] \label{thm:generating_functions}
		Let $(M, \Omega, \omega)$ be an LCS manifold with $\omega = dh$ globally. Let $\Omega_h = e^h\Omega$ be the associated symplectic form. Let $L \subset M$ be an LCS Lagrangian submanifold ($i^*\Omega = 0$). Then $L$ is also Lagrangian for $\Omega_h$ ($i^*\Omega_h = 0$). If $L$ can be locally represented as the graph of $p = dS(q)$ in Darboux coordinates $(q, p)$ for $\Omega_h$, then $S$ is a standard generating function for $L$ w.r.t. $\Omega_h$.
	\end{theorem}
	
	\begin{proof}
		We are given $i^*\Omega = 0$. Since $\Omega_h = e^h\Omega$, we compute $i^*\Omega_h = i^*(e^h\Omega) = (e^h \circ i) \cdot (i^*\Omega) = (e^h|_L) \cdot 0 = 0$. Thus, $L$ is Lagrangian for the symplectic form $\Omega_h$. Assume that in local Darboux coordinates $(q_1, \dots, q_n, p_1, \dots, p_n)$ for $\Omega_h = \sum_{k=1}^n dp_k \wedge dq_k$, the submanifold $L$ is parameterized by $q \mapsto (q, p(q))$, where $p_k = \frac{\partial S}{\partial q_k}$ for some smooth function $S(q_1, \dots, q_n)$. This means $S$ is a generating function of type I for the Lagrangian submanifold $L$ with respect to the symplectic form $\Omega_h$. The condition $i^*\Omega = 0$ is automatically satisfied because $i^*\Omega = e^{-h|_L} i^*\Omega_h$ and we already established $i^*\Omega_h = 0$. Therefore, the standard symplectic generating function $S$ for $L$ with respect to $\Omega_h$ also describes $L$ as an LCS Lagrangian for $\Omega$.
	\end{proof}
	
	\section{Concrete Example: The Kodaira-Thurston Manifold}
	
	\subsection{The Manifold and LCS Structure}
	
	To illustrate the concepts developed in this paper, particularly the non-exact case where the Lee form has a nontrivial harmonic component, we consider the Kodaira-Thurston manifold. This 4-dimensional compact nilmanifold provides a concrete example where we can perform explicit calculations of both Flux and the Twisted Calabi invariant.
	
	\begin{definition}[Kodaira-Thurston Manifold]
		The \emph{Kodaira-Thurston manifold} $M$ is the quotient of $\mathbb{R}^4$ by the discrete group $\Gamma$ generated by the transformations:
		\begin{align*}
			(x, y, z, w) &\mapsto (x+1, y, z, w), \\
			(x, y, z, w) &\mapsto (x, y+1, z, w), \\
			(x, y, z, w) &\mapsto (x, y, z+1, w), \\
			(x, y, z, w) &\mapsto (x+y, y, z, w+1).
		\end{align*}
		Topologically, $M$ is a torus bundle over a torus.
	\end{definition}
	
	We endow $M$ with the following LCS structure $(\Omega, \omega)$:
	\[
	\omega = dz, \quad \Omega = e^z dx \wedge dy + dw \wedge dz.
	\]
	
	\begin{proposition}
		The triple $(M, \Omega, \omega)$ is a compact LCS manifold. The Lee form $\omega = dz$ is harmonic and not exact (since $z$ is a coordinate on $S^1$), making this a strictly non-exact LCS structure ($l \neq 0$).
	\end{proposition}
	
	\begin{proof}
		The forms are invariant under the group action. The LCS condition holds:
		\[
		d\Omega = d(e^z dx \wedge dy) = e^z dz \wedge dx \wedge dy = -dz \wedge (e^z dx \wedge dy) = -\omega \wedge \Omega.
		\]
		Non-degeneracy is verified by the volume form:
		\[
		\frac{\Omega^2}{2!} = e^z dx \wedge dy \wedge dw \wedge dz.
		\]
	\end{proof}
	
	\subsection{Lichnerowicz Cohomology}
	
	As a prerequisite for the flux calculation, we identify the cohomology.
	
	\begin{theorem}
		For the Kodaira-Thurston manifold with $\omega = dz$, the first Lichnerowicz cohomology group is $H_\omega^1(M) \cong \mathbb{R}^2$, generated by the classes $[e^{-z}dx]_\omega$ and $[e^{-z}dy]_\omega$.
	\end{theorem}
	
	(See \cite{Banyaga2001} or similar literature for the full derivation on nilmanifolds). Note that $e^z dy$ (which appears below) represents a non-trivial cohomology class.
	
	\subsection{A Non-Hamiltonian Loop}
	
	We exhibit a loop of LCS diffeomorphisms that is not Hamiltonian, illustrating the non-triviality of the Flux group. Consider the vector field $X = \frac{\partial}{\partial x}$.
	\[
	i_X \Omega = i_{\partial_x}(e^z dx \wedge dy + dw \wedge dz) = e^z dy.
	\]
	We calculate $d^\omega(e^z dy) = d(e^z dy) + dz \wedge (e^z dy) = e^z dz \wedge dy + e^z dz \wedge dy = 0$. Thus $X$ is an LCS vector field ($L_X \Omega = d^\omega(i_X \Omega) + i_X d^\omega \Omega = 0$).	The flow $\phi_t(x, y, z, w) = (x+t, y, z, w)$ defines a loop for $t \in [0,1]$. Its flux is:
	\[
	\widetilde{\Psi}([\{\phi_t\}]) = \left[ \int_0^1 i_X \Omega \, dt \right]_\omega = [e^z dy]_\omega.
	\]
	As shown in standard computations for this manifold, $[e^z dy]_\omega \neq 0$ in $H_\omega^1(M)$. Therefore, by Proposition \ref{prop:exact_sequence}, this loop does not lie in $\Ham_\Omega(M)$. It represents a "spiral staircase" phenomenon: a loop in the diffeomorphism group that is obstructed from being Hamiltonian by the harmonic Lee form.
	
	\subsection{Twisted Calabi Invariant Calculation}
	
	In contrast to the symplectic case where "constant" Hamiltonians are trivial, the non-exact LCS setting allows for rigid, non-trivial Hamiltonian dynamics generated by constants.
	
	Consider the Hamiltonian function $H(x,y,z,w) = 1$.
	\[
	d^\omega H = d(1) + 1 \cdot dz = dz.
	\]
	We seek the vector field $V$ such that $i_V \Omega = dz$. Inspecting $\Omega = e^z dx \wedge dy + dw \wedge dz$, we find that $V = \frac{\partial}{\partial w}$ satisfies:
	\[
	i_{\partial_w} \Omega = i_{\partial_w}(dw \wedge dz) = 1 \cdot dz = dz.
	\]
	Thus, the translation in the $w$-direction, $g_t(w) = w+t$, is a \emph{Hamiltonian} LCS isotopy generated by the unique Hamiltonian $H=1$.	We compute the \textbf{Twisted Calabi Invariant} for the time-1 map $g = g_1$:
	\begin{align*}
		\Cal_\omega(g) &= \int_0^1 \left( \int_M H \, \frac{\Omega^n}{n!} \right) dt \\
		&= \int_M 1 \cdot (e^z dx \wedge dy \wedge dw \wedge dz).
	\end{align*}
	Using the coordinates of the fundamental domain $[0,1]^4$:
	\[
	\Cal_\omega(g) = \int_0^1 \int_0^1 \int_0^1 \int_0^1 e^z \, dx \, dy \, dw \, dz = \left(\int_0^1 e^z dz\right) = e - 1.
	\]
	
	\begin{remark}[Interpretation]
		The value $e-1$ is strictly non-zero.
		\begin{itemize}
			\item In a symplectic manifold, a Hamiltonian $H=1$ can usually be normalized to $H=0$, yielding zero Calabi invariant, or implies a trivial evolution.
			\item In this non-exact LCS manifold, the Hamiltonian $H=1$ is unique (Proposition \ref{prop:rigidity}) and generates a non-trivial translation $\partial_w$.
			\item The value $e-1$ represents the "conformal volume" or "topological work" required to translate along the fiber $w$ against the gradient of the Lee form's potential (conceptually $e^z$).
		\end{itemize}
		This explicitly confirms that $\Cal_\omega$ captures non-trivial dynamical information in the strictly non-exact setting.
	\end{remark}
	
	\section{Conclusion and Perspectives}
	
	This paper explored geometric aspects of locally conformally symplectic (LCS) diffeomorphisms, with a particular focus on the consequences derived from the LCS flux homomorphism introduced by S. Haller \cite{Haller2005}. A key contribution is the systematic treatment of the Hodge decomposition $\omega = dh + l$ of the Lee form, which distinguishes between exact ($l=0$) and non-exact ($l\neq 0$) LCS structures. This decomposition provides a natural measure of how "far" an LCS structure is from being symplectic and fundamentally influences the geometry.
	
	By analyzing the flux, we elucidated structural properties of the subgroup $(\ker \Phi)_0$, including the fundamental short exact sequence (Proposition \ref{prop:exact_sequence}) that positions the Hamiltonian subgroup $\Ham_\Omega(M)$ as the kernel of the flux map $\mathcal{F} : (\ker \Phi)_0 \to H_\omega^1(M)/\Delta$. We presented several theorems that extend key concepts from symplectic geometry to the LCS setting, demonstrating how the harmonic component $l$ of $\omega$ influences these structures. These include the LCS Weinstein neighborhood theorem (Theorem \ref{thm:weinstein}), LCS flux rigidity (Theorem \ref{thm:flux_rigidity}), a condition for the existence of an LCS structure on mapping tori based on vanishing flux (Theorem \ref{thm:mapping_torus}), and topological properties of $\Ham_\Omega(M)$ such as openness and the structure of flux classes (Theorem \ref{thm:deformation_rigidity}, Theorem \ref{thm:structure_ker_phi}).
	
	The Hodge decomposition clarifies which results hold generally versus those specific to the exact case:
	\begin{itemize}
		\item \textbf{General LCS results} (valid for both exact and non-exact $\omega$): Weinstein neighborhood theorem, flux rigidity, mapping torus condition, deformation rigidity, fragmentation, and the structure of $(\ker \Phi)_0$.
		\item \textbf{Exact case results} ($\omega = dh$): Calabi invariant (Theorem \ref{thm:calabi_exact}), Hofer energy and metric (Theorems \ref{thm:energy_exact}, \ref{thm:hofer_exact}), non-displaceability (Theorem \ref{thm:non_displaceability}), and generating function theory (Theorem \ref{thm:generating_functions}).
	\end{itemize}
	
	For the exact case, the link to standard symplectic geometry via conformal rescaling $\Omega_h = e^h\Omega$ allowed us to define and utilize analogues of the Calabi invariant, Hofer energy and metric, fragmentation, and generating functions for Lagrangian submanifolds. These results highlight that LCS geometry with exact Lee form behaves much like a conformally "twisted" version of symplectic geometry, inheriting many of its rigidity and structural properties.
	
	An LCS non-displaceability theorem was presented under the exactness condition $\omega = dh$, suggesting that this rigidity extends to the LCS setting in the conformally symplectic case, while the question remains open for general non-exact LCS structures where the harmonic component $l$ may allow for new phenomena.
	
	The results of this paper open many avenues for applying locally conformal symplectic (LCS) geometry to both practical and theoretical problems. In classical and quantum mechanics, they provide a natural framework for extending Hamiltonian dynamics to systems with dissipation by incorporating a nonzero Lee form into the dynamics. In the study of dynamical systems and control theory, our results offer new tools for analyzing the stability and rigidity of phase space structures under nonideal, nonconservative conditions. Additionally, in geometric topology, the construction of the extended Lee homomorphism introduces powerful invariants that can be used to classify and distinguish LCS structures up to conformal equivalence. Overall, these findings emphasize that LCS geometry is not merely a generalization of symplectic geometry but a robust and versatile framework capable of addressing essential challenges in modern physics and differential geometry.
	
	\subsection*{Future Research Directions}
	
	\begin{enumerate}
		\item \textbf{Non-exact non-displaceability:} Investigate whether Lagrangian non-displaceability holds for non-exact LCS structures or if the harmonic component allows displacement.
		
		\item \textbf{LCS Floer theory:} Develop Floer homology for non-exact LCS manifolds, adapting to the twisted differential $d^\omega$.
		
		\item \textbf{Spectral invariants:} Define and study spectral invariants for the LCS-Hofer metric using generating functions or Floer theory.
		
		\item \textbf{Higher-dimensional invariants:} Extend the twisted Calabi invariant to higher-dimensional invariants using the $d^\omega$-cohomology ring.
		
		\item \textbf{Applications to dissipative systems:} Apply LCS geometry to model mechanical systems with friction, magnetic fields, or other non-conservative forces.
	\end{enumerate}
	
	\begin{center}
		\textbf{Acknowledgments}
	\end{center}
	
	We thank the anonymous referees for their valuable comments and suggestions which helped improve this paper.
	
	\textbf{Data Availability}\\
	All data generated or analyzed during this study are included in this published article.\\
	
	\textbf{Funding}\\
	The authors declare that no funding was received for this research.\\
	
	\textbf{The authors Claim No Conflict of Interest.}

\end{document}